\documentclass[11pt]{amsart}  
\usepackage{amsmath,amssymb,amsthm,amsfonts}
\usepackage{mathrsfs}
\usepackage{xcolor}
\usepackage[all,cmtip]{xy}

\usepackage{geometry}
\usepackage{biblatex} 
\usepackage{hyperref}
\hypersetup{
    colorlinks,
    linkcolor={black},
    citecolor={black},
    urlcolor={black}
}

\theoremstyle{plain} 
\newtheorem{thm}{Theorem}[section]

\newtheorem{prop}[thm]{Proposition}
\theoremstyle{definition}

\theoremstyle{remark}
\newtheorem{rem}[thm]{Remark}

\title[Three term rational function progressions in finite fields]{Three term rational function progressions in finite fields}

\author{Guo-Dong Hong} 
\address{Department of Mathematics, California Institute of Technology, Pasadena, CA 91125, USA}
\email{ghong@caltech.edu}
\author{Zi Li Lim}
\address{Department of Mathematics, UCLA, Los Angeles, CA 90095, USA}
\email{zililim@math.ucla.edu}

\begin{document}

\maketitle

\begin{abstract}
Let $F(t),G(t)\in \mathbb{Q}(t)$ be rational functions such that $F(t),G(t)$ and the constant function $1$ are linearly independent over $\mathbb{Q}$, we prove an asymptotic formula for the number of the three term rational function progressions of the form $x,x+F(y),x+G(y)$ in subsets of $\mathbb{F}_p$. The main new ingredient is an algebraic geometry version of PET induction that bypasses Weyl's differencing. This answers a question of Bourgain and Chang.
\end{abstract}

\section{Introduction}\label{intro}
\subsection{Introduction}
Finding patterns in a dense set is a central theme of number theory and combinatorics. Let $\mathcal{A}$ be a discrete interval $[N]=\{1,2,...,N\}$ or a finite field, and $P_1,P_2,...,P_m\in \mathbb{Z}[t]$ be polynomials, define $r_{P_1,P_2,...,P_m}(\mathcal{A})$ to be the size of the largest subset of $\mathcal{A}$ that does not contain any nontrivial progression of the form
\begin{equation}
    x,x+P_1(y),x+P_2(y),...,x+P_m(y).
\end{equation}
Szemer\'{e}di \cite{Szemerédi1975} proved that any subset of the integers of positive upper density contains an arbitrarily long nontrivial arithmetic progression, or equivalently,
\begin{equation}
    r_{t,2t,...,(k-1)t}([N])=o_{k}(N)
\end{equation}
for all $k\geq 2$ via combinatorial arguments.

Bergelson and Leibman \cite{MR1325795} generalized Szemer\'{e}di's theorem to polynomial progressions. Let $P_1,P_2,...,P_m\in \mathbb{Z}[t]$ be polynomials with zero constant terms, Bergelson and Leibman proved that 
\begin{equation}
    r_{P_1,P_2,...P_m}([N])=o_{P_1,P_2,...,P_m}(N)
\end{equation}
via ergodic theory. The zero constant term assumption is a natural and convenient condition to avoid local obstructions.

Both Szemer\'{e}di's and Bergelson-Leibman's proofs are qualitative in nature. For quantitative bounds, before Szemer\'{e}di, Roth \cite{https://doi.org/10.1112/jlms/s1-28.1.104} already showed that 
\begin{equation}
    r_{t,2t}([N])=O\Bigl(\frac{N}{\log\log N}\Bigr)
\end{equation}
via Fourier analysis, but his arguments do not work for finding longer arithmetic progressions. In Gowers' celebrated papers \cites{MR1844079,MR1631259}, he obtained quantitative bounds for the sets that do not contain arbitrarily long arithmetic progressions, more precisely, for all $k\geq 2$, there exists a constant $c_{k}>0$ such that  
\begin{equation}
    r_{t,2t,...,(k-1)t}([N])=O_{k}\Bigl(\frac{N}{(\log\log N)^{c_k}}\Bigr)
\end{equation}
via introducing higher order Fourier analysis. For quantitative bounds for general polynomial progressions, the general theory is not complete yet, but see Prendiville's \cite{Prendiville2017Quantitative}, Peluse and Prendiville's \cite{peluse2022quantitative}, Peluse's \cite{peluse_2020}, Peluse and Prendiville's \cite{peluse2021polylogarithmic}, Bloom and Maynard's \cite{Bloom_Maynard_2022} and Peluse, Sah and Sawhney's \cite{peluse2023effective} results for current state of the art.

Now, we turn our attention to finding progressions in the finite fields, where new phenomena emerge. Here and throughout, let $p$ be a prime and $\mathbb{F}_p$ be the finite field with $p$ elements. One has the bound $r_{P_1,P_2,...,P_m}(\mathbb{F}_p)\leq r_{P_1,P_2,...,P_m}([p])$ \textit{a priori}, but a much stronger bound is valid in many cases. The case of two term polynomial progressions in finite fields is more straightforward and can be handled by a Fourier analytic calculation and Weil's bound. Hence, the first nontrivial polynomial progression in finite fields is the three term progression $x,x+y,x+y^2$. Bourgain and Chang \cite{MR3704938} proved that $r_{t,t^2}(\mathbb{F}_p)\ll p^{1-1/15}$. In fact, Bourgain and Chang proved an asymptotic formula: if $A$ is a subset of $\mathbb{F}_p$, then
\begin{equation}
    \#\{(x,y)\in \mathbb{F}_p^2: x,x+y,x+y^2\in A\}=\frac{|A|^3}{p}+O\bigl(p^{2/5}|A|^{3/2}\bigr).
\end{equation}
In the same paper, Bourgain and Chang posed the following three questions:
\begin{enumerate}
    \item Let $P_1,P_2\in \mathbb{Z}[t]$ be linearly independent polynomials with zero constant terms, is an asymptotic formula for the number of the progressions $x,x+P_1(y),x+P_2(y)$ in subsets of $\mathbb{F}_p$ valid?
    \item More generally, let $P_1,P_2,...,P_m\in \mathbb{Z}[t]$ be linearly independent polynomials with zero constant terms, is an asymptotic formula for the number of the progressions $x,x+P_1(y),x+P_2(y),...,x+P_m(y)$ in subsets of $\mathbb{F}_p$ valid?
    \item In a different direction, can one replace the polynomials $P_1,P_2$ in the first question $(1)$ by rational functions?
\end{enumerate}

The first question $(1)$ is answered by Peluse \cite{MR3874848} and independently by Dong, Li and Sawin \cite{MR4179774}. Peluse's theorem states that, let $P_1,P_2\in \mathbb{Z}[t]$ be polynomials with zero constant terms such that $P_1,P_2$ are linearly independent over $\mathbb{Q}$, then
\begin{equation}
    \#\{(x,y)\in \mathbb{F}_p^2: x,x+P_1(y),x+P_2(y)\in A\}=\frac{|A|^3}{p}+O_{P_1,P_2}\bigl(p^{7/16}|A|^{3/2}\bigr).
\end{equation}
for all subsets $A$ of $\mathbb{F}_p$. This implies the bound $r_{P_1,P_2}(\mathbb{F}_p)\ll_{P_1,P_2} p^{1-1/24}$. Independently, Dong, Li and Sawin proved a similar result with the error term $O_{P_1,P_2}\bigl(p^{3/8}|A|^{3/2}\bigr)$, implying $r_{P_1,P_2}(\mathbb{F}_p)\ll_{P_1,P_2} p^{1-1/12}$.  As a side note, the linear independence assumption for $P_1,P_2$ is crucial, since one cannot hope for power saving bounds for progressions involving linearly dependent polynomials, even the simplest case $x,x+y,x+2y$. Both Peluse's and Dong-Li-Sawin's methods applied deep results from algebraic geometry; Peluse's method used Deligne's bound while Dong-Li-Sawin's method used the Deligne-Katz theory.

The second question $(2)$ is also answered by Peluse \cite{MR3934588}. By introducing degree lowering method, Peluse proved that, let $P_1,P_2,...,P_m\in \mathbb{Z}[t]$ be polynomial with zero constant terms such that $P_1,P_2,...,P_m$ are linearly independent over $\mathbb{Q}$, there exists a constant $\gamma>0$ depending on $P_1,P_2,...,P_m$ such that
\begin{equation}
    \#\{(x,y)\in \mathbb{F}_p^2: x,x+P_1(y),x+P_2(y),...,x+P_m(y)\in A\}=\frac{|A|^{m+1}}{p^{m-1}}+O_{P_1,P_2,...,P_m}\bigl(p^{2-(m+1)\gamma}\bigr)
\end{equation}
for all subsets $A$ of $\mathbb{F}_p$. This implies the bound $r_{P_1,P_2,...,P_m}(\mathbb{F}_p)\ll_{P_1,P_2,...,P_m} p^{1-\gamma}$. The degree lowering method also provides a great simplification of the solution to the first question $(1)$ of Bourgain and Chang. Moreover, the degree lowering method only relies on Weil's bound for curves, as opposed to the heavy algebraic geometry machinery utilized in the methods \cites{MR3874848,MR4179774} for the three term polynomial progressions. 

For the third question $(3)$ of Bourgain and Chang, replacing the polynomials by rational functions presents new phenomena and difficulties. Unlike the case of polynomials, there is no obvious analogue of rational function progressions in integer setting, hence there is no any general bound \textit{a priori}, not even a qualitative one. The only known example is the special progression $x,x+y,x+1/y$ considered by Bourgain and Chang \cite{MR3704938}, but their proof relies on the arithmetic geometry theorem that is specific to Kloosterman sums \cite{MR3338119}. All the previous methods \cite{MR3704938,MR3874848,MR4179774,MR3934588} do not apply to the case of rational function progressions.   

In this paper, we answered this question affirmatively. To state our results, we introduce some notation first. We use $\mathbb{E}^*_{x,y}$ to denote averaging in $x,y$ variables excluding the poles of the (rational) functions being averaged. For example, if $F(t),G(t)\in \mathbb{Q}(t)$ are rational functions and $f_0,f_1,f_2:\mathbb{F}_p\longrightarrow\mathbb{C}$ are functions, then
\begin{equation*}
\mathbb{E}^*_{x,y}f_0(x)f_1(x+F(y))f_2(x+G(y))=\frac{1}{p^2}\sum_{x,y\in \mathbb{F}_p \;\text{and}\; y\neq \text{poles}}f_0(x)f_1(x+F(y))f_2(x+G(y)).
\end{equation*}
The notation $\lVert \cdot \rVert_2$ is the $L^2$-norm with respect to the uniform probability measure on $\mathbb{F}_p$, that is, $\lVert f \rVert_2=\bigl(\frac{1}{p}\sum_{x\in \mathbb{F}_p} |f(x)|^2\bigr)^{1/2}$. Our main theorem is an asymptotic formula for the counting operator for three term rational function progressions.

\begin{thm}\label{main}
    Let $F(t), G(t)\in \mathbb{Q}(t)$ be rational functions over $\mathbb{Q}$ such that $F(t), G(t)$ and the constant function $1$ are linearly independent over $\mathbb{Q}$. Then, we have the asymptotic formula
    \[
        \mathbb{E}^*_{x,y} f_0(x)f_1(x+F(y))f_2(x+G(y))=\prod_{i=0}^2\mathbb{E}_z f_i(z)+O_{F,G}\bigl(p^{-1/10}\lVert f_0 \rVert_2 \lVert f_1 \rVert_2 \lVert f_2 \rVert_2\bigr)
    \]
    for all functions $f_0,f_1,f_2:\mathbb{F}_p\longrightarrow \mathbb{C}$.  
\end{thm}

\begin{rem}
    In this paper we will frequently need to evaluate a $\mathbb{Q}$-coefficients rational function over $\mathbb{F}_p$. Here, we use the natural interpretation. Let $F(t)\in \mathbb{Q}(t)$ be a rational function and write $F(t)=a(t)/b(t)$ with $a(t),b(t)\in \mathbb{Z}[t]$ and $\gcd(a(t),b(t))=1$. For any $y\in \mathbb{F}_p$ such that $b(y)\neq 0$, define $F(y)=a(y)\cdot \bigl(b(y)\bigr)^{-1}$ where $\bigl(b(y)\bigr)^{-1}$ is the multiplicative inverse of $b(y)$ in $\mathbb{F}_p^{\times}$.
\end{rem}

\begin{rem}
    As already observed in the case of polynomial progressions, the linear independence assumption is crucial. The linear independence assumption for $F(t),G(t)$ and the constant function $1$ is reasonably natural; it includes the case that $F,G$ are linearly independent polynomials with zero constant terms. It also includes the formulation of this question in Peluse's survey \cite[Question 3.13]{peluse2023finite}.
\end{rem}

Let $f_0,f_1,f_2$ in Theorem \ref{main} be the characteristic function of a subset $A$ of $\mathbb{F}_p$, we deduce that
\begin{equation}
\begin{split}
    &\#\{(x,y)\in \mathbb{F}_p^2: y\;\text{is not a pole of}\;F\;\text{or}\;G\;\text{and}\;x,x+F(y),x+G(y)\in A\}\\
    &=\frac{|A|^3}{p}+O_{F,G}\bigl(p^{2/5}|A|^{3/2}\bigr),
\end{split}    
\end{equation}
hence we conclude the following theorem.

\begin{thm}\label{main2}
    Let $F(t), G(t)\in \mathbb{Q}(t)$ be rational functions over $\mathbb{Q}$ such that $F(t), G(t)$ and the constant function $1$ are linearly independent over $\mathbb{Q}$. There exists a constant $C_{F,G}>0$ such that if $A$ is a subset of $\mathbb{F}_p$ of size $|A|\geq C_{F,G}\;p^{1-1/15}$, then
    \[
    \#\lbrace(x,y)\in \mathbb{F}_p^2:y\;\text{is not a pole of}\;F\;\text{or}\;G\;\text{and}\;x,x+F(y),x+G(y)\in A\rbrace \gg_{F,G} \frac{|A|^3}{p}.
    \]
\end{thm}

This implies the bound $r_{F,G}(\mathbb{F}_p)\ll_{F,G}p^{1-1/15}$. Our approach to Theorem \ref{main} also provides a new solution to the first question $(1)$ of Bourgain and Chang. Compared to previous methods, the exponent in the bound $r_{F,G}(\mathbb{F}_p)\ll_{F,G}p^{1-1/15}$ is slightly better than the exponent $1-1/24$ in Peluse's method, but slightly worse than the exponent $1-1/12$ in Dong-Li-Sawin's method. Moreover, our approach has less dependence on algebraic geometry, hence for the special progressions $x,x+y,x+y^2$ and $x,x+y,x+1/y$, our proof would be an ``elementary" algebra proof.

The main new difficulty for the rational function progression problem is that differentiating a rational function does not reduce its complexity. If one wishes to apply the degree lowering method \cite{MR3934588}, then one requires Gowers norm control for the counting operator \textit{a priori}. However, Gowers norm control is done by PET induction via Weyl's differencing, which is unhelpful for rational functions. Applying Peluse's previous method for three term polynomial progressions \cite{MR3874848} or Dong-Li-Sawin's method \cite{MR4179774} also does not work, since both methods would reduce the problem to estimating multidimensional exponential sums on varieties with rational function phases, which is highly nontrivial and unknown.

The main novelty in this paper is the development of an algebraic geometry version of PET induction that bypasses Weyl's differencing. If we wanted to prove the power saving bounds directly, then we would need to estimate multidimensional exponential sums on varieties with rational function phases, which is out of reach of the current arithmetic geometry methods. However, if we only want to obtain the Gowers norm control, then it suffices to estimate the dimension of certain varieties, which is in the scope of classical algebraic geometry. We would decompose the varieties of interest into generic parts and special subvarieties. The generic parts can be dealt with complex analytic techniques, while the special subvarieties, which have extra constraints and special structures, can be handled by mod $p$ point counting. After establishing Gowers norm control, we can immediately run the degree lowering arguments. In fact, since we obtain a $U^2$-Gowers norm control, we could perform a Fourier level set decomposition of the functions to run the degree lowering arguments more efficiently, as opposed to using the regularity lemma.

A natural problem is to generalize algebraic geometry PET induction to longer rational function progressions, which is a subject we intend to return in the future. Given that all the previous methods generalize to finite fields $\mathbb{F}_q$ that may not have prime cardinality, we believe our approach also extend to all finite fields. Furthermore, it is natural to wonder about the quantitative bounds for polynomial progressions involving linearly dependent polynomials. This problem is still widely open; see Kuca's work \cite{kuca2020bounds} and Leng's works \cites{leng2022quantitative,leng2023partition} for current status.

This paper is organized as follow. In Section \ref{sec: convention}, we define the notation used in this paper. The Section \ref{sec: PET} is about the framework of the algebraic geometry PET induction and the Section \ref{sec: dim} is about the key dimension estimates required for the algebraic geometry PET induction. In Section \ref{sec: final}, we put everything together to finish the proof. The main new ideas of this paper are in Sections \ref{sec: PET} and \ref{sec: dim}.

\subsection{Acknowledgments}
The first author would like to thank his advisor David Conlon for his support and discussions. The second author would like to thank his advisor Terence Tao for inspirations, advice and support. The second author also thanks James Leng for helpful discussions.

\section{Convention}\label{sec: convention}

\subsection{Notation}
Throughout, let $p$ be a prime and $\mathbb{F}_p$ be the finite field with $p$ elements. For any finite set $\mathcal{A}$ and function $f:\mathcal{A}\longrightarrow\mathbb{C}$, let $\mathbb{E}_{x\in\mathcal{A}}f(x)=\frac{1}{|\mathcal{A}|}\sum_{x\in \mathcal{A}}f(x)$ be the average of $f$ over $\mathcal{A}$. When the summation variables are in $\mathbb{F}_p$, we would often omit the subscript $\mathbb{F}_p$ in the summation notation and write $\mathbb{E}_{x}$ instead. We use $\mathbb{E}^*$ to denote averaging excluding the poles of the (rational) functions being averaged. For example, if $F(t),G(t)\in \mathbb{Q}(t)$ are rational functions and $f_0,f_1,f_2:\mathbb{F}_p\longrightarrow\mathbb{C}$ are functions, then
\begin{equation*}
\mathbb{E}^*_{x,y}f_0(x)f_1(x+F(y))f_2(x+G(y))=\frac{1}{p^2}\sum_{x,y\in \mathbb{F}_p \;\text{and}\; y\neq \text{poles}}f_0(x)f_1(x+F(y))f_2(x+G(y)).
\end{equation*}
Similarly, we would denote the summation excluding the poles by $\sum^*$. Furthermore, we would often omit the subscript $F,G$ in Vinogradov's notation $\ll,\gg$ and Landau's notation $O,o$ since the implied constants would usually depend on the rational functions $F,G$.

\subsection{Fourier Analysis Notation}
For any function $f:\mathbb{F}_p\longrightarrow \mathbb{C}$, denote the $L^r$-norm with respect to the uniform probability measure on $\mathbb{F}_p$ by $\lVert f \rVert_r=\bigl(\frac{1}{p}\sum_{x}\lvert f(x)\rvert^r\bigr)^{1/r}$ and the $l^r$-norm with respect to the counting measure on $\mathbb{F}_p$ by $\lVert f \rVert_{l^r}=\bigl(\sum_{x}\lvert f(x)\rvert^r\bigr)^{1/r}$. Denote $e_p(x)=e^{2\pi i x/p}$ and Fourier transform $\widehat{f}(\xi)=\mathbb{E}_{x} f(x)e_p(-x\xi)$. In this convention, Parseval's identity reads $\lVert f \rVert_2=\lVert \widehat{f} \rVert_{l^2}$ and Fourier inversion reads $f(x)=\sum_{\xi}\widehat{f}(\xi)e_p(\xi x)$.

\subsection{Algebraic Geometry Terminology} By (classical) varieties, we mean any classical affine, quasi-affine, projective or quasi-projective varieties. We do not require the varieties to be irreducible, hence our affine varieties and projective varieties may be algebraic sets in some references. By abuse of terminology, we also refer to schemes that are separated and of finite type over $\mathrm{Spec}\;k$ for some field $k$ as a variety. We would often use $Y(\mathbb{C}),Z(\mathbb{F}_p),...$ to denote classical varieties and $Y,Z,...$ to denote schemes. If $X$ is a scheme over some field $k$, we also write $X(k)$ as $k$-valued points $\mathrm{Hom}_{k}(\mathrm{Spec}\;k,X)$. 

If $X$ is a variety or scheme, denote the dimension of $X$ as a variety (or scheme) by $\dim_{\mathrm{Var}}X$ (or $\dim_{\mathrm{Sch}}X$), which is the supremum of the length of the chains of irreducible closed subsets in the underlying topological space.

\section{Algebraic Geometry PET Induction}\label{sec: PET}

In this section, we establish the framework of algebraic geometry PET induction, bounding the counting operator for three term rational function progressions in terms of $U^2$-Gowers norm (equivalently, the $l^4$-norm of the Fourier transform) and the size of the $\mathbb{F}_p$-points of certain varieties. This allows us to bypass Weyl's differencing, which is unhelpful for rational functions, to obtain Gowers norm control. 

Let $F(t),G(t)\in \mathbb{Q}(t)$ be rational functions over $\mathbb{Q}$. Define the \emph{Roth variety} associated to $F(t),G(t)$ to be the (quasi-affine) variety cut out by the following system of equations in eight variables $y_1,y_2,y_3,y_4,y_5,y_6,y_7,y_8$.

\begin{gather}
    F(y_1)-G(y_1)-F(y_2)+G(y_2)-F(y_3)+G(y_3)+F(y_4)-G(y_4)=0\label{eqs: Roth1}\\
    F(y_5)-G(y_5)-F(y_6)+G(y_6)-F(y_7)+G(y_7)+F(y_8)-G(y_8)=0\label{eqs: Roth2}\\
    F(y_1)-F(y_3)-F(y_5)+F(y_7)=0\label{eqs: Roth3}\\
    F(y_2)-F(y_4)-F(y_6)+F(y_8)=0\label{eqs: Roth4}\\
    G(y_3)-G(y_4)-G(y_7)+G(y_8)=0\label{eqs: Roth5}
\end{gather}
The Roth variety could be regarded as a classical variety or a scheme over various fields, for example, the complex numbers $\mathbb{C}$ or finite fields $\mathbb{F}_p$. 

\begin{thm}\label{PET}
    Let $F(t),G(t)\in \mathbb{Q}(t)$ be rational functions over $\mathbb{Q}$ and $Y(\mathbb{F}_p)$ be $\mathbb{F}_p$-points of the Roth variety associated to $F(t),G(t)$. For all functions $f_0,f_1,f_2:\mathbb{F}_p \longrightarrow \mathbb{C}$, we have the following $U^2$-Gowers norm control:
    \[\lvert\mathbb{E}^*_{x,y} f_0(x)f_1(x+F(y))f_2(x+G(y))\rvert \leq \lVert f_0 \rVert_2^{1/2} \lVert \widehat{f_0} \rVert_{l^4}^{1/2} \lVert f_1 \rVert_2 \lVert f_2 \rVert_2\left(\frac{1}{p^3} |Y(\mathbb{F}_p)| \right)^{1/8}.\]
\end{thm}

\begin{proof}
    Fourier expand the functions $f_0,f_1,f_2$ to obtain
    \begin{equation}
        \begin{split}
            &\lvert\mathbb{E}^*_{x,y} f_0(x)f_1(x+F(y))f_2(x+G(y))\rvert\\
        &= \Bigl\lvert\sum_{n_0,n_1,n_2} \widehat{f_0}(n_0)\widehat{f_1}(n_1)\widehat{f_2}(n_2)\mathbb{E}_x e_p(n_0 x+n_1 x+n_2 x)\mathbb{E}^*_y e_p(n_1F(y)+n_2G(y))\Bigr\rvert.
        \end{split}
    \end{equation}
    
    For convenience, denote $K(n_1,n_2)=\mathbb{E}^*_y e_p(n_1F(y)+n_2G(y))$. By orthogonality of the characters, we have
    \begin{equation}
        \begin{split}
            &\Bigl\lvert\sum_{n_0,n_1,n_2} \widehat{f_0}(n_0)\widehat{f_1}(n_1)\widehat{f_2}(n_2)\mathbb{E}_x e_p(n_0 x+n_1 x+n_2 x)\mathbb{E}^*_y e_p(n_1F(y)+n_2G(y))\Bigr\rvert\\
        &=\Bigl\lvert\sum_{n_1,n_2} \widehat{f_0}(-n_1-n_2)\widehat{f_1}(n_1)\widehat{f_2}(n_2)K(n_1,n_2)\Bigr\rvert.
        \end{split}
    \end{equation}
    
    Rearrange, apply Cauchy--Schwarz inequality in $n_2$ and apply Parseval's identity to obtain
\begin{equation}
    \begin{split}
        &\Bigl\lvert\sum_{n_1,n_2} \widehat{f_0}(-n_1-n_2)\widehat{f_1}(n_1)\widehat{f_2}(n_2)K(n_1,n_2)\Bigr\rvert\\
        &=\Bigl\lvert\sum_{n_2} \widehat{f_2}(n_2)\sum_{n_1}\widehat{f_0}(-n_1-n_2)\widehat{f_1}(n_1)K(n_1,n_2)\Bigr\rvert\\
        &\leq \lVert f_2 \rVert_2 \Bigl( \sum_{n_2,n_1,m_1} \widehat{f_0}(-n_1-n_2)\overline{\widehat{f_0}}(-m_1-n_2)\widehat{f_1}(n_1)\overline{\widehat{f_1}}(m_1)K(n_1,n_2)\overline{K}(m_1,n_2)\Bigr)^{1/2}.
    \end{split}
\end{equation}

Once again, rearrange, apply Cauchy--Schwarz inequality in $n_1,m_1$ and apply Parseval's identity to obtain
\begin{equation}\label{PETeq1}
    \begin{split}
        &\lVert f_2 \rVert_2 \Bigl( \sum_{n_2,n_1,m_1} \widehat{f_0}(-n_1-n_2)\overline{\widehat{f_0}}(-m_1-n_2)\widehat{f_1}(n_1)\overline{\widehat{f_1}}(m_1)K(n_1,n_2)\overline{K}(m_1,n_2)\Bigr)^{1/2}\\
        &= \lVert f_2 \rVert_2 \Bigl( \sum_{n_1,m_1} \widehat{f_1}(n_1)\overline{\widehat{f_1}}(m_1)\sum_{n_2}\widehat{f_0}(-n_1-n_2)\overline{\widehat{f_0}}(-m_1-n_2)K(n_1,n_2)\overline{K}(m_1,n_2)\Bigr)^{1/2}\\
        &\leq \lVert f_2 \rVert_2 \lVert f_1 \rVert_2\Bigl( \sum_{n_1,m_1,n_2,m_2} \widehat{f_0}(-n_1-n_2)\overline{\widehat{f_0}}(-m_1-n_2)\overline{\widehat{f_0}}(-n_1-m_2)\widehat{f_0}(-m_1-m_2)\\
        &\quad\quad\quad\quad\quad\quad\quad\quad\quad\quad\quad\quad K(n_1,n_2)\overline{K}(m_1,n_2)\overline{K}(n_1,m_2)K(m_1,m_2)\Bigr)^{1/4}.
    \end{split}
\end{equation}

By the change of variables $n_1\leftrightarrow n_1-n_2$, $m_1\leftrightarrow m_1-n_2$, the last line in the expression \eqref{PETeq1} is equal to 
\begin{equation}\label{PETeq2}
    \begin{split}
        &\lVert f_2 \rVert_2 \lVert f_1 \rVert_2\Bigl( \sum_{n_1,m_1,n_2,m_2} \widehat{f_0}(-n_1)\overline{\widehat{f_0}}(-m_1)\overline{\widehat{f_0}}(-n_1+n_2-m_2)\widehat{f_0}(-m_1+n_2-m_2)\\
        &\quad\quad\quad\quad\quad\quad\quad\quad\quad\quad K(n_1-n_2,n_2)\overline{K}(m_1-n_2,n_2)\overline{K}(n_1-n_2,m_2)K(m_1-n_2,m_2)\Bigr)^{1/4}.
    \end{split}
\end{equation}

By the change of variables $a\leftrightarrow n_2-m_2,b\leftrightarrow n_2$, the expression \eqref{PETeq2} is equal to
\begin{equation}\label{PETeq3}
    \begin{split}
        &\lVert f_2 \rVert_2 \lVert f_1 \rVert_2\Bigl( \sum_{n_1,m_1,a,b} \widehat{f_0}(-n_1)\overline{\widehat{f_0}}(-m_1)\overline{\widehat{f_0}}(-n_1+a)\widehat{f_0}(-m_1+a)\\
        &\quad\quad\quad\quad\quad\quad\quad\quad K(n_1-b,b)\overline{K}(m_1-b,b)\overline{K}(n_1-b,b-a)K(m_1-b,b-a)\Bigr)^{1/4}.
    \end{split}
\end{equation}

Rearranging and applying Cauchy--Schwarz inequality in $n_1,m_1,a$, the expression \eqref{PETeq3} is equal to
\begin{equation}\label{PETeq4}
    \begin{split}
        &\lVert f_2 \rVert_2 \lVert f_1 \rVert_2\Bigl( \sum_{n_1,m_1,a} \widehat{f_0}(-n_1)\overline{\widehat{f_0}}(-m_1)\overline{\widehat{f_0}}(-n_1+a)\widehat{f_0}(-m_1+a)\\
        &\quad\quad\quad\quad\quad\quad\quad\quad\sum_b K(n_1-b,b)\overline{K}(m_1-b,b)\overline{K}(n_1-b,b-a)K(m_1-b,b-a)\Bigr)^{1/4}\\
        &\leq \lVert f_2 \rVert_2 \lVert f_1 \rVert_2 \Bigl( \sum_{n_1,m_1,a} \lvert \widehat{f_0}(-n_1)\overline{\widehat{f_0}}(-m_1)\overline{\widehat{f_0}}(-n_1+a)\widehat{f_0}(-m_1+a) \rvert^2 \Bigr)^{1/8}\\
        &\quad \Bigl( \sum_{n_1,m_1,a,b,b'} K(n_1-b,b)\overline{K}(m_1-b,b)\overline{K}(n_1-b,b-a)K(m_1-b,b-a)\\
        &\quad\quad\quad\quad\quad\quad \overline{K}(n_1-b',b')K(m_1-b',b')K(n_1-b',b'-a)\overline{K}(m_1-b',b'-a)\Bigr)^{1/8}.
    \end{split}
\end{equation}

We estimate the two sums above separately. First, note that
\begin{equation}\label{PETeq5}
\begin{split}
    &\sum_{n_1,m_1,a} \lvert \widehat{f_0}(-n_1)\overline{\widehat{f_0}}(-m_1)\overline{\widehat{f_0}}(-n_1+a)\widehat{f_0}(-m_1+a) \rvert^2\\
    &=\sum_{n_1,m_1} \lvert \widehat{f_0}(-n_1)\widehat{f_0}(-m_1) \rvert^2 \sum_a \lvert \widehat{f_0}(-n_1+a)\widehat{f_0}(-m_1+a) \rvert^2\\
    &\leq \sum_{n_1,m_1} \lvert \widehat{f_0}(-n_1)\widehat{f_0}(-m_1) \rvert^2 \lVert \widehat{f_0} \rVert_{l^4}^4\\
    &=\lVert f_0 \rVert_{2}^4\lVert \widehat{f_0} \rVert_{l^4}^4
\end{split}
\end{equation}
where the second to last line follows from the Cauchy--Schwarz inequality. Next, note that
\begin{equation}\label{PETeq6}
    \begin{split}
        &\sum_{n_1,m_1,a,b,b'} K(n_1-b,b)\overline{K}(m_1-b,b)\overline{K}(n_1-b,b-a)K(m_1-b,b-a)\\
        &\quad\quad\quad\quad\quad\quad \overline{K}(n_1-b',b')K(m_1-b',b')K(n_1-b',b'-a)\overline{K}(m_1-b',b'-a)\\
        &=\frac{1}{p^8}\sum_{\substack{n_1,m_1,a\\b,b'}}\quad\sideset{}{^*}\sum_{\substack{y_1,y_2,y_3,y_4\\ y_5,y_6,y_7,y_8}}\\
        &\quad\quad\quad e_p((n_1-b)F(y_1)+bG(y_1))\;e_p(-(m_1-b)F(y_2)-bG(y_2))\\
        &\quad\quad\quad e_p(-(n_1-b)F(y_3)-(b-a)G(y_3))\;e_p((m_1-b)F(y_4)+(b-a)G(y_4))\\
        &\quad\quad\quad e_p(-(n_1-b')F(y_5)-b'G(y_5))\;e_p((m_1-b')F(y_6)+b'G(y_6))\\
        &\quad\quad\quad e_p((n_1-b')F(y_7)+(b'-a)G(y_7))\;e_p(-(m_1-b')F(y_8)-(b'-a)G(y_8)).
    \end{split}
\end{equation}

This is equal to
\begin{equation}\label{PETeq7}
    \begin{split}
        &\frac{1}{p^8}\quad \sideset{}{^*}\sum_{\substack{y_1,y_2,y_3,y_4\\ y_5,y_6,y_7,y_8}}\quad\sum_{\substack{n_1,m_1,a\\b,b'}}\\
        &\quad\quad\quad e_p(b(-F(y_1)+G(y_1)+F(y_2)-G(y_2)+F(y_3)-G(y_3)-F(y_4)+G(y_4)))\\
        &\quad\quad\quad e_p(b'(F(y_5)-G(y_5)-F(y_6)+G(y_6)-F(y_7)+G(y_7)+F(y_8)-G(y_8)))\\
        &\quad\quad\quad e_p(n_1(F(y_1)-F(y_3)-F(y_5)+F(y_7)))\\
        &\quad\quad\quad e_p(m_1(-F(y_2)+F(y_4)+F(y_6)-F(y_8)))\\
        &\quad\quad\quad e_p(a(G(y_3)-G(y_4)-G(y_7)+G(y_8)))\\
        &=\frac{1}{p^8}\quad \sum_{(y_1,y_2,y_3,y_4,y_5,y_6,y_7,y_8)\in Y(\mathbb{F}_p)} p^5\\
        &=\frac{1}{p^3}|Y(\mathbb{F}_p)|
    \end{split}
\end{equation}
where the second to last equality follows from the orthogonality of the characters. Combining inequalities \eqref{PETeq4},\eqref{PETeq5} and \eqref{PETeq7} completes the proof.
\end{proof}

\begin{rem}
    In the proof of Theorem \ref{PET}, if we were to only apply Cauchy--Schwarz inequality two times and stop at \eqref{PETeq3}, then we could formulate an alternative approach to Theroem \ref{main}, conditionally on the estimates of multidimensional exponential sums on varieties with rational function phases, which is highly nontrivial and out of reach of the current arithmetic geometry methods.
\end{rem}

\section{Dimension Estimates}\label{sec: dim}

In Section \ref{sec: PET}, we establish a Gowers norm control, conditionally on estimating the size of the $\mathbb{F}_p$-points of the Roth variety. We will prove the required estimates in Theorem \ref{dim}, completing the algebraic geometry PET induction. The proof of Theorem \ref{dim} is quite long and technical, hence we would like to briefly describe the strategy first.

We would decompose the Roth variety $Y$ into five pieces $Y=Y_{\mathrm{gen}}\cup Y_{\mathrm{low}}\cup Z_{\mathrm{good}}\cup Z_{\mathrm{bad}}\cup Z_{\mathrm{low}}$ and estimate via different techniques. The variety $Y_{\mathrm{gen}}$ is a generic part of $Y$, which can be handled by complex analytic techniques. The variety $Z_{\mathrm{good}}$ is a generic component of a special subvariety of $Y$, which can be handled by complex analytic techniques too. On the other hand, the variety $Z_{\mathrm{bad}}$ is a special subvariety with extra constraints and special structures, hence we have extra inputs for mod $p$ point counting. Finally, the varieties $Y_{\mathrm{low}}$ and $Z_{\mathrm{low}}$ are of lower dimension, therefore we can control them more easily.

\begin{thm}\label{dim}
    Let $F(t),G(t)\in \mathbb{Q}(t)$ be rational functions over $\mathbb{Q}$ such that $F(t),G(t)$ and the constant function $1$ are linearly independent over $\mathbb{Q}$, and $Y(\mathbb{F}_p)$ be the $\mathbb{F}_p$-points of the Roth variety  associated to $F(t),G(t)$. Then, we have the following point counting estimates:
    \[|Y(\mathbb{F}_p)| \ll p^3.\]    
\end{thm}

\begin{proof}
    We would decompose $Y(\mathbb{F}_p)$ into various parts and estimate via different techniques accordingly. First, note that the first five columns of the Jacobian matrix of the defining equations \eqref{eqs: Roth1}, \eqref{eqs: Roth2}, \eqref{eqs: Roth3}, \eqref{eqs: Roth4}, \eqref{eqs: Roth5} of the Roth variety are
    \begin{equation*}
        \begin{bmatrix}
        F'(y_1)-G'(y_1) & -F'(y_2)+G'(y_2) & -F'(y_3)+G'(y_3) & F'(y_4)-G'(y_4) & 0 \\
        0 & 0 & 0 & 0 & F'(y_5)-G'(y_5)\\
        F'(y_1) & 0 & -F'(y_3) & 0 & -F'(y_5)\\
        0 & F'(y_2) & 0 & -F'(y_4) & 0\\
        0 & 0 & G'(y_3) & -G'(y_4) & 0        
    \end{bmatrix}.
    \end{equation*}
    
    Let $D(y_1,y_2,y_3,y_4,y_5)$ be the determinant of the matrix, that is,
    \begin{equation*}
        \bigl(G'(y_5)-F'(y_5)\bigr)\begin{vmatrix}
            F'(y_1)-G'(y_1) & -F'(y_2)+G'(y_2) & -F'(y_3)+G'(y_3) & F'(y_4)-G'(y_4)\\
            F'(y_1) & 0 & -F'(y_3) & 0\\
            0 & F'(y_2) & 0 & -F'(y_4)\\
            0 & 0 & G'(y_3) & -G'(y_4)
        \end{vmatrix}
    \end{equation*}
    which is equal to
    \begin{equation*}
        \bigl(G'(y_5)-F'(y_5)\bigr)\begin{vmatrix}
            -G'(y_1) & G'(y_2) & 0 & 0\\
            F'(y_1) & 0 & -F'(y_3) & 0\\
            0 & F'(y_2) & 0 & -F'(y_4)\\
            0 & 0 & G'(y_3) & -G'(y_4)
        \end{vmatrix}.
    \end{equation*}
    
    Hence, we have 
    \begin{equation}
        D(y_1,y_2,y_3,y_4,y_5)=(G'(y_5)-F'(y_5))\Bigl[ G'(y_1)F'(y_2)F'(y_3)G'(y_4)-F'(y_1)G'(y_2)G'(y_3)F'(y_4)\Bigr].
    \end{equation}
    
    Define
    \begin{equation}
        Y_{\mathrm{gen}}=Y\cap \{(y_1,y_2,y_3,y_4,y_5,y_6,y_7,y_8)\in \mathbb{A}^8:D(y_1,y_2,y_3,y_4,y_5)\neq 0\}
    \end{equation}
    and
    \begin{equation}
        Y_{\mathrm{sp}}=Y\cap \{(y_1,y_2,y_3,y_4,y_5,y_6,y_7,y_8)\in \mathbb{A}^8:D(y_1,y_2,y_3,y_4,y_5)=0\}.
    \end{equation}
    
    Note that we have $Y(\mathbb{F}_p)=Y_{\mathrm{gen}}(\mathbb{F}_p)\cup Y_{\mathrm{sp}}(\mathbb{F}_p)$. Informally, the variety $Y_{\mathrm{gen}}$ is a generic part of $Y$ which lies in the smooth component of $Y$, and $Y_{\mathrm{sp}}$ is a special subvariety of $Y$ cut out by $D(y_1,y_2,y_3,y_4,y_5)$.\\
    
    The associated complex analytic space of $Y_{\mathrm{gen}}(\mathbb{C})$ is a $3$-dimensional complex manifold since the Jacobian matrices have full rank and we can apply the implicit function theorem. By Noether normalization, there exists a finite surjective map $Y_{\mathrm{gen}}(\mathbb{C})\longrightarrow \mathbb{A}^d(\mathbb{C})$ between varieties, where $d$ is the dimension of $Y_{\mathrm{gen}}(\mathbb{C})$ as a variety and $\mathbb{A}^d$ is the $d$-dimensional affine space. Regard the map $Y_{\mathrm{gen}}(\mathbb{C})\longrightarrow \mathbb{A}^d(\mathbb{C})$ as a finite holomorphic map between complex manifolds. By the first theorem in \cite[Chapter 5, Section 4.1]{MR0755331}, we deduce that the dimension of $Y_{\mathrm{gen}}(\mathbb{C})$ as a variety is $3$. (We compare the dimension of a variety over $\mathbb{C}$ and the dimension of its associated complex analytic space via \textit{ad hoc} argument here, for general dimension theory of complex analytic spaces, see \cite[Appendix B]{MR0463157} and \cite[Chapter 5, Section 1]{MR0755331}.)
    
    Consider an affine $\mathbb{Z}$-model of $Y_{\mathrm{gen}}(\mathbb{C})$ by clearing out the denominators in the defining equations of the Roth variety and adding one auxiliary variable $z$ and one equation to require that the denominators are nonzero. To be more specific, write $F(t)=a(t)/b(t)$ and $G(t)=c(t)/d(t)$ with $a(t),b(t),c(t),d(t)\in \mathbb{Z}[t]$, let $I\subset \mathbb{Z}[y_1,y_2,y_3,y_4,y_5,y_6,y_7,y_8,z]$ be the ideal generated by the following polynomials:
    \begin{gather}
        \Bigl(\prod_{i=1}^4 b(y_i)d(y_i)\Bigr)\Bigl(F(y_1)-G(y_1)-F(y_2)+G(y_2)-F(y_3)+G(y_3)+F(y_4)-G(y_4)\Bigr)\\
        \Bigl(\prod_{i=5}^8 b(y_i)d(y_i)\Bigr)\Bigl(F(y_5)-G(y_5)-F(y_6)+G(y_6)-F(y_7)+G(y_7)+F(y_8)-G(y_8)\Bigr)\\
        \Bigl(b(y_1)b(y_3)b(y_5)b(y_7)\Bigr)\Bigl(F(y_1)-F(y_3)-F(y_5)+F(y_7)\Bigr)\\
        \Bigl(b(y_2)b(y_4)b(y_6)b(y_8)\Bigr)\Bigl(F(y_2)-F(y_4)-F(y_6)+F(y_8)\Bigr)\\
        \Bigl(d(y_3)d(y_4)d(y_7)d(y_8)\Bigr)\Bigl(G(y_3)-G(y_4)-G(y_7)+G(y_8)\Bigr)\\
        z\Bigl(\prod_{i=1}^8 b(y_i)d(y_i)\Bigr)\biggl[\Bigl(\prod_{i=1}^5 b(y_i)d(y_i)\Bigr)^2 D(y_1,y_2,y_3,y_4,y_5)\biggr]-1.
    \end{gather}
    
    Define affine $\mathbb{Z}$-model $Y_{\mathrm{gen}}=\mathrm{Spec}\; \mathbb{Z}[y_1,y_2,y_3,y_4,y_5,y_6,y_7,y_8,z]/I$. Note that the $\mathbb{C}$-points of $Y_{\mathrm{gen}}\times_\mathbb{Z} \mathbb{C}$ is isomorphic to $Y_{\mathrm{gen}}(\mathbb{C})$, hence we have $\dim_{\mathrm{Sch}}Y_{\mathrm{gen}}\times_\mathbb{Z} \mathbb{C}=\dim_{\mathrm{Var}}Y_{\mathrm{gen}}(\mathbb{C})=3$. By \cite[Chapter 2, Exercise 3.20.(f)]{MR0463157}, we have the invariance of dimension under base change, i.e., $\dim_{\mathrm{Sch}}Y_{\mathrm{gen}}\times_\mathbb{Z} \mathbb{Q}=\dim_{\mathrm{Sch}}Y_{\mathrm{gen}}\times_\mathbb{Z} \mathbb{C}=3$. (Note that the integral assumption in \cite[Chapter 2, Exercise 3.20.(f)]{MR0463157} can be removed if one does not require the scheme after base change is equidimensional.)
    
    Note that the structure map $Y_{\mathrm{gen}}=\mathrm{Spec}\; \mathbb{Z}[y_1,y_2,y_3,y_4,y_5,y_6,y_7,y_8,z]/I \longrightarrow \mathrm{Spec}\;\mathbb{Z}$ is of finite type and the base $\mathrm{Spec}\;\mathbb{Z}$ is irreducible, by \cite[\href{https://stacks.math.columbia.edu/tag/05F6}{Section 05F6, Lemma 37.30.1}]{stacks-project}, we deduce that the dimension of the generic fiber is the same as mod $p$ fibers for all but finitely many $p$, that is, $\dim_{\mathrm{Sch}}Y_{\mathrm{gen}}\times_\mathbb{Z} \mathbb{Q}=\dim_{\mathrm{Sch}}Y_{\mathrm{gen}}\times_\mathbb{Z} \mathbb{F}_p=3$ for all but finitely many $p$. Once again, by \cite[Chapter 2, Exercise 3.20.(f)]{MR0463157}, we have $\dim_{\mathrm{Sch}}Y_{\mathrm{gen}}\times_\mathbb{Z} \mathbb{F}_p=\dim_{\mathrm{Sch}}Y_{\mathrm{gen}}\times_\mathbb{Z} \overline{\mathbb{F}}_p$. Hence, we have $\dim_{\mathrm{Var}}Y(\overline{\mathbb{F}}_p)=3$ for all but finitely many $p$. Now, the Lang-Weil bound \cite{MR0065218} (also see \cite[Lemma 1]{Tao})  implies that $Y_{\mathrm{gen}}(\mathbb{F}_p)\ll p^3$.\\

    Next, we would like to further decompose $Y_{\mathrm{sp}}(\mathbb{F}_p)$, let $\widetilde{D}(y_1,y_2,y_3,y_4)$ denote the function
    \begin{equation}
    \frac{D(y_1,y_2,y_3,y_4,y_5)}{F'(y_1)F'(y_2)F'(y_3)F'(y_4)\bigl(G'(y_5)-F'(y_5)\bigr)}
        =\frac{G'(y_1)G'(y_4)}{F'(y_1)F'(y_4)}-\frac{G'(y_2)G'(y_3)}{F'(y_2)F'(y_3)}.
    \end{equation}
    
    Define 
    \begin{equation}
        Y_{\mathrm{low}}=Y_{\mathrm{sp}}\cap \{(y_1,y_2,y_3,y_4,y_5,y_6,y_7,y_8)\in \mathbb{A}^8:F'(y_1)F'(y_2)F'(y_3)F'(y_4)\bigl(G'(y_5)-F'(y_5)\bigr)=0\}
    \end{equation}
    and
    \begin{equation}
        Z=Y_{\mathrm{sp}}\cap \{(y_1,y_2,y_3,y_4,y_5,y_6,y_7,y_8)\in \mathbb{A}^8:\widetilde{D}(y_1,y_2,y_3,y_4)=0\}.
    \end{equation}
    
    Note that we have $Y_{\mathrm{sp}}(\mathbb{F}_p)=Y_{\mathrm{low}}(\mathbb{F}_p)\cup Z(\mathbb{F}_p)$. Informally, the variety $Y_{\mathrm{low}}(\mathbb{F}_p)$ has strong constraint $F'(y_1)F'(y_2)F'(y_3)F'(y_4)\bigl(G'(y_5)-F'(y_5)\bigr)=0$, hence this lower dimensional contribution can be bounded more easily later. We would like to further decompose $Z(\mathbb{F}_p)$. To do so, consider the Jacobian matrix of the following equations:
    \begin{gather}
        F(y_1)-G(y_1)-F(y_2)+G(y_2)-F(y_3)+G(y_3)+F(y_4)-G(y_4)=0\\
        \widetilde{D}(y_1,y_2,y_3,y_4)=0,
    \end{gather}
    which is 
    \begin{equation*}
        \begin{bmatrix}
        F'(y_1)-G'(y_1) & -F'(y_2)+G'(y_2) & -F'(y_3)+G'(y_3) & F'(y_4)-G'(y_4) \\
        \frac{G'(y_4)}{F'(y_4)}\Bigl(\frac{G'(y_1)}{F'(y_1)}\Bigr)' & -\frac{G'(y_3)}{F'(y_3)}\Bigl(\frac{G'(y_2)}{F'(y_2)}\Bigr)' & -\frac{G'(y_2)}{F'(y_2)}\Bigl(\frac{G'(y_3)}{F'(y_3)}\Bigr)' & \frac{G'(y_1)}{F'(y_1)}\Bigl(\frac{G'(y_4)}{F'(y_4)}\Bigr)' 
    \end{bmatrix}.
    \end{equation*}
    
    Let $E(y_1,y_4)$ be the determinant of the minor matrix of the Jacobian matrix consisting of the first column and fourth column, i.e.
    \begin{equation*}
        E(y_1,y_4)=\begin{vmatrix}
            F'(y_1)-G'(y_1) & F'(y_4)-G'(y_4) \\
        \frac{G'(y_4)}{F'(y_4)}\Bigl(\frac{G'(y_1)}{F'(y_1)}\Bigr)' & \frac{G'(y_1)}{F'(y_1)}\Bigl(\frac{G'(y_4)}{F'(y_4)}\Bigr)'
        \end{vmatrix}
    \end{equation*}
    which is 
    \begin{equation}
        \bigl(F'(y_1)-G'(y_1)\bigr)\frac{G'(y_1)}{F'(y_1)}\Bigl(\frac{G'(y_4)}{F'(y_4)}\Bigr)'-\bigl(F'(y_4)-G'(y_4)\bigr)\frac{G'(y_4)}{F'(y_4)}\Bigl(\frac{G'(y_1)}{F'(y_1)}\Bigr)'.
    \end{equation}
    
    Also, let $\widetilde{E}(y_1,y_4)$ denote the function
    \begin{equation}
            \frac{E(y_1,y_4)}{\Bigl(\frac{G'(y_1)}{F'(y_1)}\Bigr)'\Bigl(\frac{G'(y_4)}{F'(y_4)}\Bigr)'}
            =\frac{\bigl(F'(y_1)-G'(y_1)\bigr)\frac{G'(y_1)}{F'(y_1)}}{\Bigl(\frac{G'(y_1)}{F'(y_1)}\Bigr)'}-\frac{\bigl(F'(y_4)-G'(y_4)\bigr)\frac{G'(y_4)}{F'(y_4)}}{\Bigl(\frac{G'(y_4)}{F'(y_4)}\Bigr)'}.
    \end{equation}
    
    Define
    \begin{equation}
        Z_{good}=Z\cap \{(y_1,y_2,y_3,y_4,y_5,y_6,y_7,y_8)\in \mathbb{A}^8:E(y_1,y_4)\neq 0\},
    \end{equation}
    \begin{equation}
        Z_{bad}=Z\cap \{(y_1,y_2,y_3,y_4,y_5,y_6,y_7,y_8)\in \mathbb{A}^8:\widetilde{E}(y_1,y_4)=0\}
    \end{equation}
    and
    \begin{equation}
        Z_{low}=Z\cap \Bigl\{(y_1,y_2,y_3,y_4,y_5,y_6,y_7,y_8)\in \mathbb{A}^8:\Bigl(\frac{G'(y_1)}{F'(y_1)}\Bigr)'\Bigl(\frac{G'(y_4)}{F'(y_4)}\Bigr)'=0\Bigr\}.
    \end{equation}
    
    Note that we have $Z(\mathbb{F}_p)=Z_{\mathrm{good}}(\mathbb{F}_p)\cup Z_{\mathrm{bad}}(\mathbb{F}_p)\cup Z_{\mathrm{low}}(\mathbb{F}_p)$. Informally, the variety $Z_{\mathrm{good}}(\mathbb{F}_p)$ is a generic part of $Z(\mathbb{F}_p)$ which satisfies good property, the variety $Z_{\mathrm{bad}}(\mathbb{F}_p)$ is a special subvariety cut out by $\widetilde{E}(y_1,y_4)=0$, and $Z_{\mathrm{low}}(\mathbb{F}_p)$ is a lower dimensional contribution. It remains to bound $|Z_{\mathrm{good}}(\mathbb{F}_p)|,|Z_{\mathrm{bad}}(\mathbb{F}_p)|,|Z_{\mathrm{low}}(\mathbb{F}_p)|$ and $|Y_{\mathrm{low}}(\mathbb{F}_p)|$.\\
    
    We start from bounding $|Z_{\mathrm{good}}(\mathbb{F}_p)|$. Consider the auxiliary variety $X\subset \mathbb{A}^4$ defined by the following equations
    \begin{gather}
        F(y_1)-G(y_1)-F(y_2)+G(y_2)-F(y_3)+G(y_3)+F(y_4)-G(y_4)=0\label{dimeq1}\\
        \widetilde{D}(y_1,y_2,y_3,y_4)=0\label{dimeq2}\\
        E(y_1,y_4)\neq 0.
    \end{gather}
    
    The associated complex analytic space of $X(\mathbb{C})$ is a $2$-dimensional complex manifold since the condition $E(y_1,y_4)\neq 0$ guarantees that the Jacobian matrices of the equations \eqref{dimeq1}, \eqref{dimeq2} have full rank and we can apply the implicit function theorem. Similar to the dimension comparison between $Y_{\mathrm{gen}}(\mathbb{C})$ and its associated complex analytic space, this implies that the dimension of $X(\mathbb{C})$ as a variety is $2$. Moreover, similar to the previous arguments regarding $Y_{\mathrm{gen}}$, we conclude that $X(\mathbb{F}_p)\ll p^2$.
    
    For all points $(y_1,y_2,y_3,y_4,y_5,y_6,y_7,y_8)\in Z_{\mathrm{good}}(\mathbb{F}_p)$, the coordinates $(y_1,y_2,y_3,y_4)$ lie on $X(\mathbb{F}_p)$, hence there are at most $O(p^2)$ choices of $y_1,y_2,y_3,y_4$. Besides, there are at most $p$ choices of $y_5$, hence there are at most $O(p^3)$ choices of $y_1,y_2,y_3,y_4,y_5$. For any given $y_1,y_2,y_3,y_4,y_5$, there are only $O(1)$ choices of $y_7$ by equation \eqref{eqs: Roth3}. For any given $y_1,y_2,y_3,y_4,y_5,y_7$, there are only $O(1)$ choices of $y_8$ by equation \eqref{eqs: Roth5}. For any given $y_1,y_2,y_3,y_4,y_5,y_7,y_8$, there are only $O(1)$ choices of $y_6$ by equation \eqref{eqs: Roth4}. Therefore, we have $Z_{\mathrm{good}}(\mathbb{F}_p)\ll p^3$.\\
    
    Next, we turn our attention to bounding $|Z_{\mathrm{bad}}(\mathbb{F}_p)|$. We would first show that the equation $\widetilde{E}(y_1,y_4)=0$ is a nontrivial constraint by showing 
    \begin{equation}\label{dimeq6}
        \frac{\bigl(F'(y_1)-G'(y_1)\bigr)\frac{G'(y_1)}{F'(y_1)}}{\Bigl(\frac{G'(y_1)}{F'(y_1)}\Bigr)'}
    \end{equation}
    is nonconstant. Suppose not, we have
    \begin{equation}
        c \Bigl(\frac{G'(y_1)}{F'(y_1)}\Bigr)'=\bigl(F'(y_1)-G'(y_1)\bigr)\frac{G'(y_1)}{F'(y_1)}
    \end{equation}
    for some constant $c$. If $c=0$, then we have $F'(y_1)-G'(y_1)=0$ or $G'(y_1)/F'(y_1)=0$, which both contradict the linear independence of $F,G$ and the constant function $1$. If $c\neq 0$, then we have
    \begin{equation}\label{dimeq3}
        \Bigl(\frac{G'(y_1)}{F'(y_1)}\Bigr)'-\frac{F'(y_1)-G'(y_1)}{c}\frac{G'(y_1)}{F'(y_1)}=0.
    \end{equation}
    
    Multiply integrating factor to equation \eqref{dimeq3}, we have
    \begin{equation}\label{dimeq4}
        \biggl[ \exp\Bigl(-\frac{F(y_1)-G(y_1)}{c}\Bigr) \frac{G'(y_1)}{F'(y_1)} \biggr]'=0,
    \end{equation}
    integrate both side of \eqref{dimeq4}, we deduce
    \begin{equation}\label{dimeq5}
        \exp\Bigl(-\frac{F(y_1)-G(y_1)}{c}\Bigr) \frac{G'(y_1)}{F'(y_1)}=\widetilde{c}
    \end{equation}
    for some constant $\widetilde{c}$. Since $G'(y_1)/F'(y_1)$ and $\bigl(F(y_1)-G(y_1)\bigr)/c$ are rational functions, the equality \eqref{dimeq5} can only happen when $G'(y_1)/F'(y_1)=0$ or $\bigl(F(y_1)-G(y_1)\bigr)/c$ is a constant, which both contradict the linear independence of $F,G$ and the constant function $1$. Hence, the rational function \eqref{dimeq6} is not a constant.
    
    Now, for all points $(y_1,y_2,y_3,y_4,y_5,y_6,y_7,y_8)\in Z_{\mathrm{bad}}(\mathbb{F}_p)$, the coordinates $y_2,y_4,y_5$ have at most $p^3$ choices. For any given $y_2,y_4,y_5$, there are at most $O(1)$ choices of $y_1$ by equation $\widetilde{E}(y_1,y_4)=0$. For any given $y_1,y_2,y_4,y_5$, there are at most $O(1)$ choices of $y_3$ by equation \eqref{eqs: Roth1}. For any given $y_1,y_2,y_3,y_4,y_5$, there are at most $O(1)$ choices of $y_7$ by equation \eqref{eqs: Roth3}. For any given $y_1,y_2,y_3,y_4,y_5,y_7$, there are at most $O(1)$ choices of $y_8$ by equation \eqref{eqs: Roth5}. For any given $y_1,y_2,y_3,y_4,y_5,y_7,y_8$, there are at most $O(1)$ choices of $y_6$ by equation \eqref{eqs: Roth4}. Therefore, we have $|Z_{\mathrm{bad}}(\mathbb{F}_p)|\ll p^3$.\\
    
Finally, we estimate $|Y_{\mathrm{low}}(\mathbb{F}_p)|$ and $|Z_{\mathrm{low}}(\mathbb{F}_p)|$. These varieties have strong constraints which allow us to bound their cardinality more easily. We first decompose $Y_{\mathrm{low}}(\mathbb{F}_p)$. Define
\begin{equation}
    Y_{\mathrm{low}}^{(1)}=Y_{\mathrm{low}}\cap \{(y_1,y_2,y_3,y_4,y_5,y_6,y_7,y_8)\in \mathbb{A}^8:F'(y_1)=0\},
\end{equation}
\begin{equation}
    Y_{\mathrm{low}}^{(2)}=Y_{\mathrm{low}}\cap \{(y_1,y_2,y_3,y_4,y_5,y_6,y_7,y_8)\in \mathbb{A}^8:F'(y_2)=0\},
\end{equation}
\begin{equation}
    Y_{\mathrm{low}}^{(3)}=Y_{\mathrm{low}}\cap \{(y_1,y_2,y_3,y_4,y_5,y_6,y_7,y_8)\in \mathbb{A}^8:F'(y_3)=0\},
\end{equation}
\begin{equation}
    Y_{\mathrm{low}}^{(4)}=Y_{\mathrm{low}}\cap \{(y_1,y_2,y_3,y_4,y_5,y_6,y_7,y_8)\in \mathbb{A}^8:F'(y_4)=0\}
\end{equation}
and
\begin{equation}
    Y_{\mathrm{low}}^{(5)}=Y_{\mathrm{low}}\cap \{(y_1,y_2,y_3,y_4,y_5,y_6,y_7,y_8)\in \mathbb{A}^8:\bigl(G'(y_5)-F'(y_5)\bigr)=0\}.
\end{equation}

    Note that we have $Y_{\mathrm{low}}(\mathbb{F}_p)=Y_{\mathrm{low}}^{(1)}(\mathbb{F}_p)\cup Y_{\mathrm{low}}^{(2)}(\mathbb{F}_p)\cup Y_{\mathrm{low}}^{(3)}(\mathbb{F}_p)\cup Y_{\mathrm{low}}^{(4)}(\mathbb{F}_p)\cup Y_{\mathrm{low}}^{(5)}(\mathbb{F}_p)$. For all $(y_1,y_2,y_3,y_4,y_5,y_6,y_7,y_8)\in Y_{\mathrm{low}}^{(1)}(\mathbb{F}_p)$, the coordinates $y_2,y_4,y_5$ have at most $p^3$ choices. For any given $y_2,y_4,y_5$, there are at most $O(1)$ choices of $y_1$ by equation $F'(y_1)=0$. For any given $y_1,y_2,y_4,y_5$, there are at most $O(1)$ choices of $y_3$ by equation \eqref{eqs: Roth1}. For any given $y_1,y_2,y_3,y_4,y_5$, there are at most $O(1)$ choices of $y_7$ by equation \eqref{eqs: Roth3}. For any given $y_1,y_2,y_3,y_4,y_5,y_7$, there are at most $O(1)$ choices of $y_8$ by equation \eqref{eqs: Roth5}. For any given $y_1,y_2,y_3,y_4,y_5,y_7,y_8$, there are at most $O(1)$ choices of $y_6$ by equation \eqref{eqs: Roth4}. Therefore, we have $|Y_{\mathrm{low}}^{(1)}(\mathbb{F}_p)|\ll p^3$. Similarly, we have $|Y_{\mathrm{low}}^{(2)}(\mathbb{F}_p)|,|Y_{\mathrm{low}}^{(3)}(\mathbb{F}_p)|,|Y_{\mathrm{low}}^{(4)}(\mathbb{F}_p)|,|Y_{\mathrm{low}}^{(5)}(\mathbb{F}_p)|\ll p^3$ and hence $|Y_{\mathrm{low}}(\mathbb{F}_p)|\ll p^3$. A similar decomposing and counting argument would show that $|Z_{\mathrm{low}}(\mathbb{F}_p)|\ll p^3$ too. This completes the proof.
\end{proof}

\section{Final Arguments}\label{sec: final}

In this section, we would put everything together to complete the proof of Theorem \ref{main}. After establishing Gowers norm control, we could immediately run the degree lowering arguments \cite{MR3934588}. In fact, since we obtain a $U^2$-Gowers norm control, we would perform a Fourier level set decomposition of the functions to run the degree lowering arguments more efficiently, as opposed to using the regularity lemma.

We first start with the case of two term rational function progressions, which serves as the base case of the induction. As observed by Peluse \cite{MR3934588}, one should consider the counting operator twisted by the characters for the purpose of induction. The proof of Proposition \ref{base} is essentially the same as \cite[Lemma 3.2]{MR3934588}, except that we apply Bombieri's bound instead of Weil's bound. Denote the characteristic function by $1_{\xi=0}$, that is, $1_{\xi=0}=1$ when $\xi=0$ and $1_{\xi=0}=0$ when $\xi\neq 0$.

\begin{prop}\label{base}
    Let $F(t), R(t)\in \mathbb{Q}(t)$ be rational functions over $\mathbb{Q}$ such that $F(t), R(t)$ and the constant function $1$ are linearly independent over $\mathbb{Q}$. Then, we have the asymptotic formula
    \[
        \mathbb{E}^*_{x,y} f_0(x)f_1(x+F(y))e_p(\xi R(y))=1_{\xi=0}\prod_{i=0}^1\mathbb{E}_z f_i(z)+O_{F,R}\bigl(p^{-1/2}\lVert f_0 \rVert_2 \lVert f_1 \rVert_2 \bigr)
    \]
    for all functions $f_0,f_1:\mathbb{F}_p\longrightarrow \mathbb{C}$ and $\xi \in \mathbb{F}_p$.
\end{prop}

\begin{proof}
    Fourier expand the functions $f_0,f_1$ to obtain
    \begin{equation}
        \begin{split}
            &\mathbb{E}^*_{x,y} f_0(x)f_1(x+F(y))e_p(\xi R(y))\\
            &=\mathbb{E}^*_{x,y} \sum_{n_0,n_1} \widehat{f_0}(n_0)\widehat{f_1}(n_1)e_p(n_0x+n_1x)e_p(n_1F(y)+\xi R(y))\\
            &=\sum_{n_1} \widehat{f_0}(-n_1)\widehat{f_1}(n_1)\mathbb{E}^*_{y} e_p(n_1F(y)+\xi R(y)).
        \end{split}
    \end{equation}
    
    By linear independence of $F,R$ and the constant function $1$, the function $n_1F(y)+\xi R(y)$ is nonconstant unless $n_1=\xi=0$. When $n_1\neq 0$ or $\xi\neq 0$, we have Bombieri's bound on single variable exponential sums in finite fields with rational function phases \cite{MR0200267} (also see \cite[Proposition 2.2]{MR3704938}), that is,
    \begin{equation}
        \mathbb{E}^*_{y} e_p(n_1F(y)+\xi R(y))\ll p^{-1/2}.
    \end{equation}
    
    Combining Bombieri's bound, the Cauchy--Schwarz inequality and Parseval's identity, we have
    \begin{equation}
        \sum_{n_1\neq 0} \widehat{f_0}(-n_1)\widehat{f_1}(n_1)\mathbb{E}^*_{y} e_p(n_1F(y)+\xi R(y))\ll p^{-1/2}\lVert f_0 \rVert_2 \lVert f_1 \rVert_2. 
    \end{equation}
    
    Note that the main term comes from the term with $n_1=0$, that is,
    \begin{equation}
        \widehat{f_0}(0)\widehat{f_1}(0)\mathbb{E}^*_{y} e_p(\xi R(y))=1_{\xi=0}\prod_{i=0}^1\mathbb{E}_z f_i(z)+O\bigl(p^{-1/2}\lVert f_0 \rVert_2 \lVert f_1 \rVert_2 \bigr).
    \end{equation}
    
    This completes the proof.
\end{proof}

Next, we show that the counting operator can be controlled by the dual function.

\begin{prop}\label{dual}
    Let $F(t),G(t)\in \mathbb{Q}(t)$ be rational functions over $\mathbb{Q}$ and $f_0,f_1,f_2:\mathbb{F}_p\longrightarrow \mathbb{C}$ be functions. Define the dual function $\mathcal{D}=\mathcal{D}_{f_1,f_2}$ by
    \[\mathcal{D}(x)=\mathbb{E}^*_y f_1(x+F(y))f_2(x+G(y)).\]
    Then, we have
    \[\lvert \mathbb{E}^*_{x,y} f_0(x)f_1(x+F(y))f_2(x+G(y)) \rvert \leq \lVert \widehat{f_0} \rVert_{l^1}\lVert \widehat{\mathcal{D}} \rVert_{l^{\infty}}.\]
\end{prop}

\begin{proof}
    By the definition of dual function and Parseval's identity, we have
    \begin{equation}
        \lvert \mathbb{E}^*_{x,y} f_0(x)f_1(x+F(y))f_2(x+G(y)) \rvert =\lvert \mathbb{E}_{x} f_0(x)\mathcal{D}(x)\rvert=\Bigl\lvert \sum_{m} \widehat{f_0}(-m)\widehat{\mathcal{D}}(m)\Bigr\rvert,
    \end{equation}
    which is bounded by $\lVert \widehat{f_0} \rVert_{l^1}\lVert \widehat{\mathcal{D}} \rVert_{l^{\infty}}$ by applying H\"{o}lder's inequality.
\end{proof}

The following proposition says that the dual function is controlled by shorter progressions.

\begin{prop}\label{dual2}
    Let $F(t),G(t)\in \mathbb{Q}(t)$ be rational functions over $\mathbb{Q}$ such that $F(t),G(t)$ and the constant function $1$ are linearly independent over $\mathbb{Q}$, and $f_1,f_2:\mathbb{F}_p\longrightarrow \mathbb{C}$ be functions with $\mathbb{E}_z f_2(z)=0$. Define the dual function $\mathcal{D}=\mathcal{D}_{f_1,f_2}$ by
    \[\mathcal{D}(x)=\mathbb{E}^*_y f_1(x+F(y))f_2(x+G(y)).\]
    Then, we have
    \[\lVert \widehat{\mathcal{D}} \rVert_{l^{\infty}} \ll p^{-1/2}\lVert f_1 \rVert_2\lVert f_2 \rVert_2.\]
\end{prop}

\begin{proof}
    Note that 
    \begin{equation}
    \begin{split}
        \widehat{\mathcal{D}}(\xi)=&\mathbb{E}_{x}\mathbb{E}^*_{y} f_1(x+F(y))f_2(x+G(y))e_p(-x\xi)\\
        =& \mathbb{E}^*_{x,y} f_1(x)e_p(-x\xi)f_2(x+G(y)-F(y))e_p(\xi F(y)).
    \end{split}
    \end{equation}
    To finish the proof, apply Proposition \ref{base} and note that the main term is $0$ due to the assumption $\mathbb{E}_z f_2(z)=0$.
\end{proof}

Finally, we have all ingredients to finish the proof of Theorem \ref{main}.

\begin{proof}[Proof of Theorem \ref{main}]
We will use convenient notation $\Lambda_{F,G}(f_0,f_1,f_2)$ to denote the counting operator $\mathbb{E}^*_{x,y} f_0(x)f_1(x+F(y))f_2(x+G(y))$. First, we start with a preliminary reduction. Let $f_2'=f_2-\mathbb{E}_z f_2(z)$ be the balanced version of the function $f_2$, so that the mean of $f_2'$ is $0$. Note that we have
\begin{equation}
    \begin{split}
        \Lambda_{F,G}(f_0,f_1,f_2)&=\Lambda_{F,G}(f_0,f_1,\mathbb{E}_z f_2(z))+\Lambda_{F,G}(f_0,f_1,f_2')\\
        &=\Bigl(\mathbb{E}_z f_2(z)\Bigr)\mathbb{E}^*_{x,y}f_0(x)f_1(x+F(y))+\Lambda_{F,G}(f_0,f_1,f_2')\\
        &=\prod_{i=0}^2\mathbb{E}_z f_i(z)+O\bigl(p^{-1/2}\lVert f_0 \rVert_2 \lVert f_1 \rVert_2 \lVert f_2 \rVert_2\bigr)+\Lambda_{F,G}(f_0,f_1,f_2'),
    \end{split}
\end{equation}
where the last equality follows from Proposition \ref{base}. It remains to bound $\Lambda_{F,G}(f_0,f_1,f_2')$.

Next, we would construct a Fourier level set decomposition of $f_0$. Let $\varepsilon>0$ be a parameter to be chosen later, define
\begin{equation}
    g(\xi)=\begin{cases}
        \widehat{f_0}(\xi), & \lvert\widehat{f_0}(\xi)\rvert>\frac{\lVert f_0 \rVert_2}{\varepsilon p^{1/2}}\\
        0, & \text{otherwise}
    \end{cases}
\end{equation}
and
\begin{equation}
    h(\xi)=\begin{cases}
        \widehat{f_0}(\xi), & \lvert\widehat{f_0}(\xi)\rvert\leq\frac{\lVert f_0 \rVert_2}{\varepsilon p^{1/2}}\\
        0, & \text{otherwise}
    \end{cases}.
\end{equation}

Note that we have $\widehat{f_0}=g+h$. By Parseval's identity, we have $l^2$-bounds $\lVert g \rVert_{l^2}\leq \lVert f_0 \rVert_{2}$ and $\lVert h \rVert_{l^2}\leq \lVert f_0 \rVert_{2}$. Moreover, we have good $l^1$-bound for $g$, that is,
\begin{equation}
    \lVert g \rVert_{l^1}=\sum_{\xi} \lvert g(\xi) \rvert \leq \sum_{\xi} \lvert g(\xi) \rvert \cdot \lvert \widehat{f_0}(\xi) \rvert \biggl(\frac{\lVert f_0 \rVert_2}{\varepsilon p^{1/2}}\biggr)^{-1}\leq \varepsilon p^{1/2} \lVert f_0 \rVert_2.
\end{equation}

Furthermore, we have good $l^4$-bound for $h$, that is,
\begin{equation}
    \lVert h \rVert_{l^4}\leq \lVert h \rVert_{l^{\infty}}^{1/2}\lVert h \rVert_{l^2}^{1/2}\leq \biggl(\frac{\lVert f_0 \rVert_2}{\varepsilon p^{1/2}}\biggr)^{1/2}\lVert h \rVert_{l^2}^{1/2}\leq \varepsilon^{-1/2}p^{-1/4}\lVert f_0 \rVert_{2}.
\end{equation}

Decompose $f_0=\check{g}+\check{h}$, where $\check{g},\check{h}$ are the inverse Fourier transform of $g,h$ respectively.

To finish the proof, we would use dual function control to bound $\Lambda_{F,G}(\check{g},f_1,f_2')$ and algebraic geometry PET induction to bound $\Lambda_{F,G}(\check{h},f_1,f_2')$. By Proposition \ref{dual}, Proposition \ref{dual2} and the $l^1$-bound for $g$, we have
\begin{equation}
    \lvert\Lambda_{F,G}(\check{g},f_1,f_2')\rvert\leq \lVert g \rVert_{l^1}\lVert \widehat{\mathcal{D}} \rVert_{l^{\infty}}\ll p^{-1/2}\lVert g \rVert_{l^1}\lVert f_1 \rVert_2\lVert f_2 \rVert_2\leq \varepsilon\lVert f_0 \rVert_2\lVert f_1 \rVert_2\lVert f_2 \rVert_2.
\end{equation}

By Theorem \ref{PET}, Theorem \ref{dim} and the $l^4$-bound for $h$, we have
\begin{equation}
    \lvert\Lambda_{F,G}(\check{h},f_1,f_2')\rvert\ll \lVert \check{h} \rVert_2^{1/2} \lVert h \rVert_{l^4}^{1/2} \lVert f_1 \rVert_2 \lVert f_2 \rVert_2\leq \varepsilon^{-1/4}p^{-1/8}\lVert f_0 \rVert_2\lVert f_1 \rVert_2\lVert f_2 \rVert_2.
\end{equation}

Hence, we have
\begin{equation}
\begin{split}
    \Lambda_{F,G}(f_0,f_1,f_2')&=\Lambda_{F,G}(\check{g},f_1,f_2')+\Lambda_{F,G}(\check{h},f_1,f_2')\\
    & \ll \varepsilon\lVert f_0 \rVert_2\lVert f_1 \rVert_2\lVert f_2 \rVert_2+\varepsilon^{-1/4}p^{-1/8}\lVert f_0 \rVert_2\lVert f_1 \rVert_2\lVert f_2 \rVert_2. 
\end{split}
\end{equation}

Choosing $\varepsilon=p^{-1/10}$ will optimize the result and complete the proof.
\end{proof}

\printbibliography 

@article {MR3934588,
    AUTHOR = {Peluse, Sarah},
     TITLE = {On the polynomial {S}zemer\'{e}di theorem in finite fields},
   JOURNAL = {Duke Math. J.},
  FJOURNAL = {Duke Mathematical Journal},
    VOLUME = {168},
      YEAR = {2019},
    NUMBER = {5},
     PAGES = {749--774},
      ISSN = {0012-7094,1547-7398},
   MRCLASS = {11B30 (11B25)},
  MRNUMBER = {3934588},
MRREVIEWER = {Pierre-Yves\ Bienvenu},
       DOI = {10.1215/00127094-2018-0051},
       URL = {https://doi.org/10.1215/00127094-2018-0051},
}

@article {MR1325795,
    AUTHOR = {Bergelson, V. and Leibman, A.},
     TITLE = {Polynomial extensions of van der {W}aerden's and
              {S}zemer\'{e}di's theorems},
   JOURNAL = {J. Amer. Math. Soc.},
  FJOURNAL = {Journal of the American Mathematical Society},
    VOLUME = {9},
      YEAR = {1996},
    NUMBER = {3},
     PAGES = {725--753},
      ISSN = {0894-0347,1088-6834},
   MRCLASS = {11B25 (05D10 28D05 54H20)},
  MRNUMBER = {1325795},
MRREVIEWER = {Pierre\ Michel},
       DOI = {10.1090/S0894-0347-96-00194-4},
       URL = {https://doi.org/10.1090/S0894-0347-96-00194-4},
}

@article {MR3704938,
    AUTHOR = {Bourgain, J. and Chang, M.-C.},
     TITLE = {Nonlinear {R}oth type theorems in finite fields},
   JOURNAL = {Israel J. Math.},
  FJOURNAL = {Israel Journal of Mathematics},
    VOLUME = {221},
      YEAR = {2017},
    NUMBER = {2},
     PAGES = {853--867},
      ISSN = {0021-2172,1565-8511},
   MRCLASS = {11B30 (37A45)},
  MRNUMBER = {3704938},
MRREVIEWER = {Ben\ Joseph\ Green},
       DOI = {10.1007/s11856-017-1577-9},
       URL = {https://doi.org/10.1007/s11856-017-1577-9},
}

@article {MR4179774,
    AUTHOR = {Dong, Dong and Li, Xiaochun and Sawin, Will},
     TITLE = {Improved estimates for polynomial {R}oth type theorems in
              finite fields},
   JOURNAL = {J. Anal. Math.},
  FJOURNAL = {Journal d'Analyse Math\'{e}matique},
    VOLUME = {141},
      YEAR = {2020},
    NUMBER = {2},
     PAGES = {689--705},
      ISSN = {0021-7670,1565-8538},
   MRCLASS = {11B30},
  MRNUMBER = {4179774},
MRREVIEWER = {Robert\ F.\ Tichy},
       DOI = {10.1007/s11854-020-0113-8},
       URL = {https://doi.org/10.1007/s11854-020-0113-8},
}

@article {MR1844079,
    AUTHOR = {Gowers, W. T.},
     TITLE = {A new proof of {S}zemer\'{e}di's theorem},
   JOURNAL = {Geom. Funct. Anal.},
  FJOURNAL = {Geometric and Functional Analysis},
    VOLUME = {11},
      YEAR = {2001},
    NUMBER = {3},
     PAGES = {465--588},
      ISSN = {1016-443X,1420-8970},
   MRCLASS = {11B25 (11K38 11K45)},
  MRNUMBER = {1844079},
MRREVIEWER = {Hillel\ Furstenberg},
       DOI = {10.1007/s00039-001-0332-9},
       URL = {https://doi.org/10.1007/s00039-001-0332-9},
}

@article {MR1631259,
    AUTHOR = {Gowers, W. T.},
     TITLE = {A new proof of {S}zemer\'{e}di's theorem for arithmetic
              progressions of length four},
   JOURNAL = {Geom. Funct. Anal.},
  FJOURNAL = {Geometric and Functional Analysis},
    VOLUME = {8},
      YEAR = {1998},
    NUMBER = {3},
     PAGES = {529--551},
      ISSN = {1016-443X,1420-8970},
   MRCLASS = {11B25 (11N13)},
  MRNUMBER = {1631259},
       DOI = {10.1007/s000390050065},
       URL = {https://doi.org/10.1007/s000390050065},
}

@article {MR3874848,
    AUTHOR = {Peluse, Sarah},
     TITLE = {Three-term polynomial progressions in subsets of finite
              fields},
   JOURNAL = {Israel J. Math.},
  FJOURNAL = {Israel Journal of Mathematics},
    VOLUME = {228},
      YEAR = {2018},
    NUMBER = {1},
     PAGES = {379--405},
      ISSN = {0021-2172,1565-8511},
   MRCLASS = {11B30 (05B99)},
  MRNUMBER = {3874848},
MRREVIEWER = {Wolfgang\ A.\ Schmid},
       DOI = {10.1007/s11856-018-1768-z},
       URL = {https://doi.org/10.1007/s11856-018-1768-z},
}

@article {MR3338119,
    AUTHOR = {Fouvry, \'{E}tienne and Kowalski, Emmanuel and Michel,
              Philippe},
     TITLE = {A study in sums of products},
   JOURNAL = {Philos. Trans. Roy. Soc. A},
  FJOURNAL = {Philosophical Transactions of the Royal Society A.
              Mathematical, Physical and Engineering Sciences},
    VOLUME = {373},
      YEAR = {2015},
    NUMBER = {2040},
     PAGES = {20140309, 26},
      ISSN = {1364-503X,1471-2962},
   MRCLASS = {11L05},
  MRNUMBER = {3338119},
MRREVIEWER = {S.\ W.\ Graham},
       DOI = {10.1098/rsta.2014.0309},
       URL = {https://doi.org/10.1098/rsta.2014.0309},
}

@article {MR0200267,
    AUTHOR = {Bombieri, Enrico},
     TITLE = {On exponential sums in finite fields},
   JOURNAL = {Amer. J. Math.},
  FJOURNAL = {American Journal of Mathematics},
    VOLUME = {88},
      YEAR = {1966},
     PAGES = {71--105},
      ISSN = {0002-9327,1080-6377},
   MRCLASS = {14.48 (10.41)},
  MRNUMBER = {200267},
MRREVIEWER = {D.\ J.\ Lewis},
       DOI = {10.2307/2373048},
       URL = {https://doi.org/10.2307/2373048},
}

@article {MR0065218,
    AUTHOR = {Lang, Serge and Weil, Andr\'{e}},
     TITLE = {Number of points of varieties in finite fields},
   JOURNAL = {Amer. J. Math.},
  FJOURNAL = {American Journal of Mathematics},
    VOLUME = {76},
      YEAR = {1954},
     PAGES = {819--827},
      ISSN = {0002-9327,1080-6377},
   MRCLASS = {14.0X},
  MRNUMBER = {65218},
MRREVIEWER = {B.\ Segre},
       DOI = {10.2307/2372655},
       URL = {https://doi.org/10.2307/2372655},
}

@book {MR0463157,
    AUTHOR = {Hartshorne, Robin},
     TITLE = {Algebraic geometry},
    SERIES = {Graduate Texts in Mathematics},
    VOLUME = {No. 52},
 PUBLISHER = {Springer-Verlag, New York-Heidelberg},
      YEAR = {1977},
     PAGES = {xvi+496},
      ISBN = {0-387-90244-9},
   MRCLASS = {14-01},
  MRNUMBER = {463157},
MRREVIEWER = {Robert\ Speiser},
}

@book {MR0755331,
    AUTHOR = {Grauert, Hans and Remmert, Reinhold},
     TITLE = {Coherent analytic sheaves},
    SERIES = {Grundlehren der mathematischen Wissenschaften [Fundamental
              Principles of Mathematical Sciences]},
    VOLUME = {265},
 PUBLISHER = {Springer-Verlag, Berlin},
      YEAR = {1984},
     PAGES = {xviii+249},
      ISBN = {3-540-13178-7},
   MRCLASS = {32-02 (32Bxx 32C30)},
  MRNUMBER = {755331},
MRREVIEWER = {Daniel\ Barlet},
       DOI = {10.1007/978-3-642-69582-7},
       URL = {https://doi.org/10.1007/978-3-642-69582-7},
}

@misc{stacks-project,
  author       = {The {Stacks project authors}},
  title        = {The Stacks project},
  howpublished = {\url{https://stacks.math.columbia.edu}},
  year         = {2023},
}

@misc{Tao,
  author       = {Tao, Terence},
  title        = {The Lang-Weil Bound},
  howpublished = {\url{https://terrytao.wordpress.com/2012/08/31/the-lang-weil-bound/}},
  year         = {2023},
}

@article{Szemerédi1975,
author = {Szemerédi, E.},
journal = {Acta Arithmetica},
language = {eng},
number = {1},
pages = {199-245},
title = {On sets of integers containing k elements in arithmetic progression},
url = {http://eudml.org/doc/205339},
volume = {27},
year = {1975},
}

@article{https://doi.org/10.1112/jlms/s1-28.1.104,
author = {Roth, K. F.},
title = {On Certain Sets of Integers},
journal = {Journal of the London Mathematical Society},
volume = {s1-28},
number = {1},
pages = {104-109},
doi = {https://doi.org/10.1112/jlms/s1-28.1.104},
url = {https://londmathsoc.onlinelibrary.wiley.com/doi/abs/10.1112/jlms/s1-28.1.104},
eprint = {https://londmathsoc.onlinelibrary.wiley.com/doi/pdf/10.1112/jlms/s1-28.1.104},
year = {1953}
}

@article{Prendiville2017Quantitative,
	journal={Discrete Analysis},
	doi={10.19086/da.1282},
	title={Quantitative bounds in the polynomial Szemerédi theorem: the homogeneous case},
	author={Prendiville, Sean},
	date={2017-02-21},
	year=2017,
	month=2,
	day=21,
}

@article{peluse_2020, title={Bounds for sets with no polynomial progressions}, volume={8}, DOI={10.1017/fmp.2020.11}, journal={Forum of Mathematics, Pi}, publisher={Cambridge University Press}, author={Peluse, Sarah}, year={2020}, pages={e16}}

@misc{peluse2023finite,
      title={Finite field models in arithmetic combinatorics -- twenty years on}, 
      author={Sarah Peluse},
      year={2023},
      eprint={2312.08100},
      archivePrefix={arXiv},
      primaryClass={math.NT}
}

@article{kuca2020bounds,
    AUTHOR = {Kuca, Borys},
     TITLE = {Further bounds in the polynomial {S}zemer\'{e}di theorem over
              finite fields},
   JOURNAL = {Acta Arith.},
  FJOURNAL = {Acta Arithmetica},
    VOLUME = {198},
      YEAR = {2021},
    NUMBER = {1},
     PAGES = {77--108},
      ISSN = {0065-1036,1730-6264},
   MRCLASS = {11B30},
  MRNUMBER = {4214350},
MRREVIEWER = {Robert\ F.\ Tichy},
       DOI = {10.4064/aa200218-9-6},
       URL = {https://doi.org/10.4064/aa200218-9-6},
}

@misc{leng2022quantitative,
      title={A Quantitative Bound For Szemer\'edi's Theorem for a Complexity One Polynomial Progression over $\mathbb{Z}/N\mathbb{Z}$}, 
      author={James Leng},
      year={2022},
      eprint={2205.05540},
      archivePrefix={arXiv},
      primaryClass={math.NT}
}

@misc{peluse2023effective,
      title={Effective bounds for Roth's theorem with shifted square common difference}, 
      author={Sarah Peluse and Ashwin Sah and Mehtaab Sawhney},
      year={2023},
      eprint={2309.08359},
      archivePrefix={arXiv},
      primaryClass={math.NT}
}

@misc{leng2023partition,
      title={The partition rank vs. analytic rank problem for cyclic groups I. Equidistribution for periodic nilsequences}, 
      author={James Leng},
      year={2023},
      eprint={2306.13820},
      archivePrefix={arXiv},
      primaryClass={math.NT}
}

@misc{peluse2022quantitative,
      title={Quantitative bounds in the nonlinear Roth theorem}, 
      author={Sarah Peluse and Sean Prendiville},
      year={2022},
      eprint={1903.02592},
      archivePrefix={arXiv},
      primaryClass={math.NT}
}

@article{peluse2021polylogarithmic,
      AUTHOR = {Peluse, Sarah and Prendiville, Sean},
     TITLE = {A polylogarithmic bound in the nonlinear {R}oth theorem},
   JOURNAL = {Int. Math. Res. Not. IMRN},
  FJOURNAL = {International Mathematics Research Notices. IMRN},
      YEAR = {2022},
    NUMBER = {8},
     PAGES = {5658--5684},
      ISSN = {1073-7928,1687-0247},
   MRCLASS = {11B25},
  MRNUMBER = {4410765},
MRREVIEWER = {Ben\ Joseph\ Green},
       DOI = {10.1093/imrn/rnaa261},
       URL = {https://doi.org/10.1093/imrn/rnaa261},

}

@article{Bloom_Maynard_2022, title={A new upper bound for sets with no square differences}, volume={158}, DOI={10.1112/S0010437X22007679}, number={8}, journal={Compositio Mathematica}, author={Bloom, Thomas F. and Maynard, James}, year={2022}, pages={1777–1798}}

\end{document}